\documentclass{article}

\usepackage[top=30truemm,bottom=30truemm,left=25truemm,right=25truemm]{geometry}

\usepackage[nobysame,alphabetic]{amsrefs}

\usepackage{amsmath,amsthm,amssymb,graphicx,enumitem}
\usepackage{tikz}
\usetikzlibrary{arrows,decorations.pathmorphing,backgrounds,positioning,fit,petri,patterns}

\usepackage[all]{xy}

\usepackage[vcentermath]{youngtab}

\theoremstyle{definition}
\newtheorem{defi}{Definition}[section]

\theoremstyle{plain}
\newtheorem{lemm}[defi]{Lemma}

\newtheorem{prop}[defi]{Proposition}

\theoremstyle{remark}
\newtheorem{rema}[defi]{Remark}
\newtheorem*{rema*}{Remark}

\newtheorem*{exam*}{Example}

\newcommand{\kks}[2][k]{g^{(#1)}_{#2}}

\newcommand{\PP}{\mathcal{P}}

\newcommand{\la}{\lambda}
\newcommand{\La}{\Lambda}

\newcommand{\de}{\delta}
\newcommand{\DE}[1]{\de\left[#1\right]}

\newcommand{\Z}{\mathbb{Z}}

\newcommand{\lra}{\longrightarrow}

\newcommand{\sm}{\setminus}

\newcommand{\ti}{\tilde}

\renewcommand{\emptyset}{\varnothing}

\newcommand{\tikzitem}[2][1.0]{
	\begin{tikzpicture}[scale=#1]
		#2
	\end{tikzpicture}
}

\newcommand{\tikzbox}[3][]{
	\draw [#1] (#3,#2) rectangle +(-1,-1);
}

\newcommand{\tikzydiag}[2][]{
	\foreach \li [count=\i] in {#2}{
		\foreach \j in {1,...,\li}{
			\tikzbox[#1]{\i}{\j};
		}
	}
}
\newcommand{\tikzyfill}[2][gray]{
	\foreach \li [count=\i] in {#2}{
		\foreach \j in {1,...,\li}{
			\fill[#1] (\j,\i) rectangle +(-1,-1);
		}
	}
}

\usepackage{todonotes}
\newcommand{\Todo}[1]{}

\usepackage{hyperref}

\usepackage{stmaryrd}

\usepackage{mathtools}
\mathtoolsset{showonlyrefs=true}

\newcommand{\gs}[1]{\widetilde{g}_{#1}}

\newcommand{\Gs}[1]{\widetilde{G}_{#1}}

\newcommand{\hatone}{\hat{1}}
\newcommand{\wt}{\widetilde}
\newcommand{\wh}{\widehat}

\newcommand{\hs}[2][]{\mathrm{HS}_{#1}(#2)}
\newcommand{\whs}[2][]{\wh{\mathrm{HS}}_{#1}(#2)}

\newcommand{\Gr}{\mathrm{Gr}}

\newcommand{\Kth}[1][k,n]{K^*(\Gr(#1))}
\newcommand{\Kho}[1][k,n]{K_*(\Gr(#1))}

\newcommand{\mc}{\mathcal}
\newcommand{\mcO}{\mc{O}}
\newcommand{\mcI}{\mc{I}}

\newcommand{\skewyd}[3][1.0]{
	\tikzitem[#1]{
		\tikzyfill{#3}
		\tikzydiag{#2}
	}
}

\newcommand{\dgYoungDiagSize}{.2}
\newcommand{\yd}[1]{
	\ifthenelse{\equal{#1}{}}{
		\emptyset
	}
	{
		\skewyd[\dgYoungDiagSize]{#1}{}
	}
}

\newcommand{\TitleContent}
{On the Pieri rules of stable and dual stable Grothendieck polynomials}

\title{\TitleContent}
\author{Motoki Takigiku}
\date{\today}

\AtEndDocument{\bigskip{\footnotesize%
  \textsc{Graduate School of Mathematical Sciences, the University of Tokyo, Japan} \par  
  \textit{E-mail address}: \texttt{takigiku@ms.u-tokyo.ac.jp} 
}}

\newcommand{\AbstractContent}{
	We give an explanation for the Pieri coefficients for the stable and dual stable Grothendieck polynomials; their non-leading terms are obtained by taking an alternating sum of meets (or joins) of their leading terms.
}

\newcommand{\AcknowledgementContent}{
    The author would like to thank
    Takeshi Ikeda
    for communicating to him the idea of considering 
    the class of the structure sheaves of Schubert varieties
    in the $K$-homology of the affine Grassmannian
    when the author did a study on $\kks{\la}$,
    which is where
    the idea of taking the sum $\sum_{\mu\subset\la}g_\la$ originally came from.
    The author is also grateful to
    Itaru Terada
    for many valuable discussions and comments.
    This work was supported by
    the Program for Leading Graduate
    Schools, MEXT, Japan.
}

\begin{document}

\maketitle

\begin{abstract}
	\AbstractContent
\end{abstract}


\section{Introduction}

The stable Grothendieck polynomials $G_\la$
and the dual stable Grothendieck polynomials $g_\la$
are certain families of inhomogeneous symmetric functions
parametrized by interger partitions $\la$.
They are certain $K$-theoretic deformations of the Schur functions
and dual to each other via the Hall inner product.

Historically
the stable Grothendieck polynomials (parametrized by permutations)
were introduced by Fomin and Kirillov \cite{MR1394950}
as a stable limit of the Grothendieck polynomials of
Lascoux--Sch\"utzenberger \cite{MR686357}.
In \cite{MR1946917} Buch gave a 
combinatorial formula for the stable Grothendieck polynomials $G_\la$
for partitions using so-called set-valued tableaux,
and showed that 
their span $\bigoplus_{\la\in\PP}\Z G_\la$ is a bialgebra and
its certain quotient ring
is isomorphic to the $K$-theory of the Grassmannian
$\mathrm{Gr}=\mathrm{Gr}(k,\mathbb{C}^{n})$.

The dual stable Grothendieck polynomials $g_\la$
were
introduced by Lam and Pylyavskyy \cite{MR2377012}
as generating functions of reverse plane partitions,
and shown to be the dual basis for 
$G_\la$
via the Hall inner product.
They also showed there that
$g_\la$ represent the $K$-homology classes of ideal sheaves of 
the boundaries of Schubert varieties in the Grassmannians.

The Pieri rule for $G_\la$ \eqref{eq:G:Pieri} was given in \cite{MR1763950},
and that for $g_\la$ \eqref{eq:g_Pieri_hs} was given in \cite{MR1946917} 
as a formula for coproduct structure constants of $G_\la$.
Both formulas involve certain binomials coefficients,
and we show in this paper that these coefficients are
the values of the M\"obius functions of certain posets of horizontal strips (Lemma \ref{theo:mobius})
and hence the Pieri formulas are written as alternating sums of meets/joins of the leading terms
(Proposition \ref{theo:Pieris} and \ref{theo:G:Pieri}).
We also explain in Section \ref{sect:Ggsums} that 
the linear map $g_\la\mapsto\sum_{\mu\subset\la}g_\mu$ ($=:\gs{\la}$) is a ring automorphism and
the linear map $G_\la\mapsto\sum_{\mu\supset\la}G_\mu$ ($=:\Gs{\la}$) is a multiplication map.
With these bases, the Pieri rules are rewritten as certain multiplicity-free sums (\eqref{eq:gsum} and \eqref{eq:Gs:Pieri:G} in Proposition \ref{theo:Pieris} and \ref{theo:G:Pieri}).

\subsection*{Acknowledgment}
\AcknowledgementContent

\section{Stable and dual stable Grothendieck polynomials }\label{sect:Prel::gla}

For basic definitions for symmetric functions,
see for instance \cite[Chapter I]{MR1354144}.

Let
$\PP$ be the set of integer partitions.
For partitions $\la,\mu\in\PP$,
the inclusion $\la\subset\mu$ means $\la_i\le\mu_i$ for all $i$,
and $\la\cap\mu$ and $\la\cup\mu$ ($\in\PP$) are given by
$(\la\cap\mu)_i=\min(\la_i,\mu_i)$ and
$(\la\cup\mu)_i=\max(\la_i,\mu_i)$ for all $i$.
In other words, $\cap$ and $\cup$ are the meet and join of the poset $(\PP,\subset)$.

Let $\La$ be the ring of symmetric functions,
namely consisting of all symmetric formal power series
in variable $x=(x_1,x_2,\dots)$
with bounded degree.
Let $\wh\La$ be its completion,
consisting of all symmetric formal power series
(with unbounded degree).

In \cite[Theorem 3.1]{MR1946917}
Buch gave a combinatorial description of
the {\em stable Grothendieck polynomial} $G_\la$
as a (signed) generating function of so-called {\em set-valued tableaux}.
We do not review the detail here
and just recall some of its properties:
$G_\la\in\wh\La$ (although $G_\la\notin\La$),
$G_\la$ is an infinite linear combination of the Schur functions $\{s_\mu\}_{\mu\in\PP}$
and its lowest degree component is $s_\la$
(hence $\wh\La=\prod_{\la\in\PP}\Z G_\la$).
Moreover
the span
$\bigoplus_\la \Z G_\la$ ($\subset\wh\La$) is a bialgebra,
in particular the expansion of 
the product
$G_\mu G_\nu = \sum_{\la} c^{\la}_{\mu\nu} G_\la$ and
the coproduct
$\Delta(G_\la) = \sum_{\mu,\nu} d^{\la}_{\mu\nu} G_\mu\otimes G_\nu$
are finite.
The {\em dual stable Grothendieck polynomial} $g_{\la}$ (for $\la\in\PP$)
is defined 
in \cite{MR2377012}
as the generating function of so-called
\textit{reverse plane partitions} of shape $\la$.
It is also shown there that
$g_{\la}\in\La$ 
and $g_{\la}$ has 
the highest degree component $s_{\la}$
and thus forms a $\Z$-basis of $\La$.
Moreover $g_\la$ is dual to $G_\la$:
it holds 
$(G_\la,g_\mu)=\delta_{\la\mu}$
where
$(\,,)\colon\wh\La\times\La\lra\Z$ is
the Hall inner product.
Hence the product (resp.\,coproduct) structure constants for $G_\la$ coincide with
the coproduct (resp.\,product) structure constants for $g_\la$:
it holds
$g_\mu g_\nu = \sum_{\la} d^{\la}_{\mu\nu} g_\la$ and
$\Delta(g_\la) = \sum_{\mu,\nu} c^{\la}_{\mu\nu} g_\mu\otimes g_\nu$.

\subsection{Pieri rules}

The (row) Pieri formula for $G_\la$ was given by Lenart \cite[Theorem 3.2]{MR1763950}:
for any partition $\la\in\PP$ and integer $a\ge 0$,
\begin{equation}\label{eq:G:Pieri}
	G_{(a)} G_\la = 
	\sum_{\mu/\la \text{: horizontal strip}}
	(-1)^{|\mu/\la|-a}
	\binom{r(\mu/\la)-1}{|\mu/\la|-a}
	G_\mu,
\end{equation}
where $r(\mu/\la)$ denotes the number of the rows in the skew shape $\mu/\la$.
Subsequently,
the (row) Pieri formula for $g_\la$ is given
in \cite[Corollary 7.1]{MR1946917}
(as a formula for $d^{\mu}_{\la,(a)}$, the coproduct structure constants for $G_\la$):
\begin{equation}
g_{(a)} g_\la 
=
\sum_{\mu/\la\text{$:$ horizontal strip}}
(-1)^{a-|\mu/\la|}
\binom{r(\la/\bar\mu)}{a-|\mu/\la|}
g_\mu,
\label{eq:g_Pieri_hs}
\end{equation}
where $\bar\mu=(\mu_2,\mu_3,\dots)$.

\subsection{Their sums}\label{sect:Ggsums}

For $\la\in\PP$
we let
\[
	\gs{\la} = \sum_{\mu\subset\la} g_\mu \ (\in\La),
	\qquad
	\Gs{\la} = \sum_{\mu\supset\la} G_\mu \ (\in\wh\La).
\]
It is known (see \cite[Section 8]{MR1946917}) that
$(1-G_1)^{-1} = \sum_{\la\in\PP}G_\la$ and
\begin{equation}\label{eq:H(1):G}
(1-G_1)^{-1} G_\la = \sum_{\mu\supset\la} G_\mu \ \big(= \Gs{\la}\big).
\end{equation}
It is also easy to see that $1-G_1 = \sum_{i\ge 0}(-1)^{i} e_i$
and hence $(1-G_1)^{-1}=\sum_{i\ge 0}h_i$ ($=:H(1)$),
where $e_i$ and $h_i$ are the elementary and complete symmetric functions.

Recall the notation
$F^\perp(f) = \sum (F,f_1) f_2$ for $F\in\wh\La$, $f\in\La$ and
$\Delta(f)=\sum f_1\otimes f_2$ with the Sweedler notation,
and that the multiplication map by $F$ is the dual map of $F^\perp$.
From this, \eqref{eq:H(1):G} and 
that $G_\la$ and $g_\la$ are dual,
we see that
$H(1)^\perp(g_\la)=\gs{\la}$.
Besides it is known (see \cite[Chapter 1.5, Example 29]{MR1354144}) that
$H(1)^\perp(f(x_1,x_2,\cdots)) = f(1,x_1,x_2,\cdots)$ for any $f\in\La$,
and hence $H(1)^\perp$ is a ring morphism.
Since $F^\perp G^\perp=(GF)^\perp$ in general,
that $H(1)$ is invertible implies that so is $H(1)^\perp$.
Hence we have
\begin{prop}
	\label{theo:H(1):perp}
	Let
	$H(1) = \sum_{i\ge 0} h_i$.
	The map
	$H(1)^\perp\colon\La\lra\La$
	is a ring automorphism and
	\begin{align}
	\gs{\la}(x) &= H(1)^\perp (g_\la(x)) = g_\la(1,x).
	\label{eq:H(1):g}
	\end{align}
	where we write $f(x)=f(x_1,x_2,\cdots)$ and 
	$f(1,x)=f(1,x_1,x_2,\cdots)$.
\end{prop}
Note that 
we can directly show $\gs{\la}(x_1,x_2,\cdots)=g_\la(1,x_1,x_2,\cdots)$
from the fact that $g_\la$ is a generating function of reverse plane partitions;
see \cite{Takigiku_dualstable2} for more details.
As seen in Section \ref{sect:geom} below,
$\gs{\la}$
correspond to the classes in $K$-homology of the structure sheaves of Schubert varieties in the Grassmannian.

\subsection{$K$-(co)homology of Grassmannians} \label{sect:geom}

	We recall geometric interpretations of $G_\la$ and $g_\la$.
	Let 
	$\Gr(k,n)$ be the Grassmannian of $k$-dimensional subspaces of $\mathbb{C}^n$,
	$R=(n-k)^k$ the rectangle of shape $(n-k)\times k$,
	and $\mcO_\la$ (for $\la\subset R$) the structure sheaves of Schubert varieties of $\Gr(k,n)$.
	The $K$-theory $\Kth$,
	the Grothendieck group of algebraic vector bundles on $\Gr(k,n)$,
	has a basis 
	$\{[\mcO_\la]\}_{\la\subset R}$, 
	and the surjection
	$\bigoplus_{\la\in\PP} \Z G_\la \lra \Kth = \bigoplus_{\la\subset R} \Z [\mcO_\la]$
	that maps $G_\la$ to $[\mcO_\la]$
	(which is considered as $0$ if $\la\not\subset R$)
	is an algebra homomorphism \cite{MR1946917}.

	There is another basis of $\Kth$ consisting of the classes $[\mcI_\la]$
	of ideal sheaves of boundaries of Schubert varieties.
	In \cite[Section 8]{MR1946917} it is shown that
	the bases $\{[\mcO_\la]\}_{\la\subset R}$ and $\{[\mcI_\la]\}_{\la\subset R}$ relates to each other by
	$[\mcO_\la] = \sum_{\la\subset\mu\subset R} [\mcI_\mu]$
	and that they are dual:
	more precisely $([\mcO_\la],[\mcI_{\ti\mu}])=\de_{\la\mu}$ where
	$\ti\mu=(n-k-\mu_k,\cdots,n-k-\mu_1)$ is the rotated complement of $\mu\subset R$
	and the pairing $(\,,)$ is defined by
	$(\alpha,\beta) = \rho_*(\alpha\otimes\beta)$ 
	where
	$\rho_*$
	is the pushforward to a point.

	%
	The $K$-homology $\Kho$,
	the Grothendieck group of coherent sheaves,
	is naturally isomorphic to $\Kth$. 
	Lam and Pylyavskyy proved in \cite[Theorem 9.16]{MR2377012} that the surjection
	$\La=\bigoplus_{\la\in\PP} \Z g_\la \lra \Kho = \bigoplus_{\mu\subset R} \Z [\mcI_\mu]$
	that maps $g_\la$ to $[\mcI_{\ti\la}]$
	(which is considered as $0$ if $\la\subset R$)
	identifies the coproduct and product on $\La$ with
	the pushforwards of
	the diagonal embedding map
	and
	the direct sum map.
	
	Since $\mu\subset\la\iff\tilde\mu\supset\tilde\la$, 
	under this identification we see that
	$\sum_{\mu\subset\la}g_\mu\in\La$ corresponds to $[\mcO_{\tilde\la}]\in\Kho$.

\section{Description for the Pieri coefficients}\label{sect:gs_Pieri}

In this section
we give an explanation for the Pieri coefficients for $G_\la$ \eqref{eq:G:Pieri} and $g_\la$ \eqref{eq:g_Pieri_hs};
their non-leading terms
(higher-degree terms for the case of $G_\la$; lower-degree terms for the case of $g_\la$)
are obtained by taking an alternating sum of meets/joins of the leading terms (\eqref{eq:g:Pieri:altsum} and \eqref{eq:G:Pieri:altsum}).
Another equivalent description is that
the product $\Gs{\la} G_{(a)}$ (resp.\,$\gs{\la} \gs{(a)}$)
is expanded into a certain multiplicity-free sum of $G_\mu$ (resp.\,$g_\mu$) 
(\eqref{eq:gsum} and \eqref{eq:Gs:Pieri:G}).

The key fact is that the coefficients in the Pieri rule
\eqref{eq:G:Pieri} and \eqref{eq:g_Pieri_hs}
are values of the M\"obius functions of certain posets of horizontal strips over $\la$:
for $\la\in\PP$ and $a\in\Z_{>0}$,
let
\begin{gather}
	\hs{\la}=\{\mu\in\PP\mid\text{$\mu/\la$ is a horizontal strip}\}, \\
	\hs[\le a]{\la}=\{\mu\in\hs{\la} \mid |\mu/\la|\le a\}, \qquad
	\whs[\le a]{\la}= \hs[\le a]{\la} \sqcup \{\hat{1}\}, \\
	\hs[\ge a]{\la}=\{\mu\in\hs{\la} \mid |\mu/\la|\ge a\}, \qquad
	\whs[\ge a]{\la}= \hs[\ge a]{\la} \sqcup \{\hat{0}\}.
\end{gather}
Here $\hat{0}$ and $\hat{1}$ are the minimum and maximal elements.
For a poset $P$, let $\mu_P$ denote its M\"obius function
(see Appendix \ref{sect:Prel::Mobius}).
Then we have
\begin{lemm}\label{theo:mobius}
	\noindent $(1)$
	For any $\mu\in\hs[\ge a]{\la}$,
	we have
	$c^{\mu}_{\la,(a)} = - \mu_{\whs[\ge a]{\la}}(\hat{0}, \mu)$.
	That is,
	\begin{gather}
	\sum_{\substack{
			\mu\supset\nu\in\hs[\ge a]{\la}
		}
	}
	c^{\nu}_{\la,(a)}
	= 1.
	\label{eq:c_sum}
	\end{gather}
	
	\noindent $(2)$
	For any $\mu\in\hs[\le a]{\la}$,
	we have
	$d^{\mu}_{\la,(a)} = - \mu_{\whs[\le a]{\la}}(\mu,\hat{1})$.
	That is,
	\begin{gather}
	\sum_{\substack{
			\mu\subset\nu\in\hs[\le a]{\la} 
	}}
	d^{\nu}_{\la,(a)}
	= 1.
	\label{eq:d_sum}
	\end{gather}
	
\end{lemm}

Before proving Lemma \ref{theo:mobius}
we show the following propositions.
Let $\la^{(1)}, \la^{(2)}, \cdots$ be the list of all horizontal strips over $\la$ of size $a$.
Then

\begin{prop}\label{theo:Pieris}
	We have
	\begin{align}
		\gs{(a)} \gs{\la}
		&=
			\sum_{\mu\subset\la^{(i)} \text{ for $\exists i$}}
				g_{\mu}
				\label{eq:gsum} \\
		&=
			\sum_{i} \gs{\mu^{(i)}}
			- \sum_{i<j} \gs{\mu^{(i)}\cap\mu^{(j)}}
			+ \sum_{i<j<k} \gs{\mu^{(i)}\cap\mu^{(j)}\cap\mu^{(k)}}
			- \cdots,
			\label{eq:gs:Pieri:altsum}
	\end{align}
	and
	\begin{align}
		g_{(a)} g_\la 
		&= 
			\sum_{i} g_{\la^{(i)}}
			- \sum_{i<j} g_{\la^{(i)}\cap\la^{(j)}}
			+ \sum_{i<j<k} g_{\la^{(i)}\cap\la^{(j)}\cap\la^{(k)}}
			- \cdots. 
			\label{eq:g:Pieri:altsum}
	\end{align}

\end{prop}

\begin{prop}\label{theo:G:Pieri}
	We have
	\begin{align}
		G_{(a)} \Gs{\la}
		&=	\sum_{\mu\supset\la^{(i)} \text{ for $\exists i$}}
			G_\mu \label{eq:Gs:Pieri:G} \\
		&=	\sum_{i} \Gs{\la^{(i)}}
			- \sum_{i < j} \Gs{\la^{(i)} \cup \la^{(j)}}
			+ \sum_{i < j < k} \Gs{\la^{(i)} \cup \la^{(j)} \cup \la^{(k)}}
			- \cdots, \label{eq:Gs:Pieri:altsum}
	\end{align}
	and
	\begin{equation}
		G_{(a)} G_{\la}
		=	\sum_{i} G_{\la^{(i)}}
		- \sum_{i < j} G_{\la^{(i)} \cup \la^{(j)}}
		+ \sum_{i < j < k} G_{\la^{(i)} \cup \la^{(j)} \cup \la^{(k)}}
		- \cdots
		\label{eq:G:Pieri:altsum}
	\end{equation}
\end{prop}

Note that the left-hand side of \eqref{eq:Gs:Pieri:G} is
not $\Gs{(a)}\Gs{\la}$ but $G_{(a)}\Gs{\la}$
while that of \eqref{eq:gsum} is $\gs{(a)}\gs{\la}$,
reflecting the fact that the map $G_\la\mapsto\Gs{\la}$ is a module morphism
while $g_\la\mapsto\gs{\la}$ is a ring morphism.

\begin{rema}
	\eqref{eq:gsum} and \eqref{eq:gs:Pieri:altsum} are mere specializations
	of corresponding results for \emph{affine dual stable Grothendieck polynomials} $\kks{\la}$
	shown in \cite{TakigikuKkSchurSum},
	but here we give another proof since it is easier and also applicable to $G_\la$.
	It is also notable that in the affine case
	(that is, for $\kks{\la}$),
	equations of the form \eqref{eq:gsum} and \eqref{eq:gs:Pieri:altsum} hold
	but \eqref{eq:g:Pieri:altsum} does not.
	
	In an earlier version of this paper
	\footnote{arXiv:1806.06369v2}
	there was an exposition of the proof of \eqref{eq:gsum} and \eqref{eq:gs:Pieri:altsum} that is adopted from \cite{TakigikuKkSchurSum} and optimized for the non-affine case,
	and by using this and the argument of Lemma \ref{theo:mobius}
	the fact that $g_\la\mapsto\gs{\la}$ is a ring morphism
	was derived.
	Later, a simpler proof for this was found (as given in Section \ref{sect:Ggsums}) and the exposition became unnecessary and therefore has been removed.
\end{rema}

\begin{proof}[Proof of Proposition \ref{theo:Pieris}]
	The right-hand sides of \eqref{eq:gsum} and \eqref{eq:gs:Pieri:altsum} are equal
	by the Inclusion-Exclusion Principle,
	and \eqref{eq:gs:Pieri:altsum} and \eqref{eq:g:Pieri:altsum} are equivalent
	by Proposition \ref{theo:H(1):perp}.
	
	Let $P$ be the order ideal of $\PP$ generated by
	$\{\la^{(1)}, \la^{(2)}, \cdots\}$
	(i.e.\,the set of $\mu\in\PP$ satisfying $\mu\subset\la^{(i)}$ for some $i$)
	and $\wh{P}=P\sqcup\{\hatone\}$ where $\hatone$ is the maximum element.
	Note that
	$\{\la^{(1)}, \la^{(2)}, \cdots\}$
	is the set of coatoms in $\wh{P}$ and
	$\whs[\le a]{\la}$ ($\subset\wh{P}$) is closed under meet.
	Then
	\begin{alignat}{2}
		\gs{\la} \gs{(a)} 
		&= \sum_{\nu} d^{\nu}_{\la,(a)} \gs{\nu} 
			& \qquad&\text{(\eqref{eq:g_Pieri_hs} and Proposition \ref{theo:H(1):perp})} \\
		&= - \sum_{\nu} \mu_{\whs[\le a]{\la}}(\nu,\hat{1}) \gs{\nu} 
			& &\text{(Lemma \ref{theo:mobius} (2))} \\
		&= - \sum_{\nu} \mu_{\wh{P}}(\nu,\hat{1}) \gs{\nu}
			& &\text{(Lemma \ref{theo:MobiusLem} (3))} \\
		&= \sum_{\mu\in P} g_\mu.
			& &\text{(Lemma \ref{theo:MobiusLem} (1))}
	\end{alignat}
	Hence \eqref{eq:gsum} follows.
\end{proof}

\begin{proof}[Proof of Proposition \ref{theo:G:Pieri}]
	Similarly to Proposition \ref{theo:Pieris},
	the equivalence of 
	\eqref{eq:Gs:Pieri:G}, 
	\eqref{eq:Gs:Pieri:altsum} and
	\eqref{eq:G:Pieri:altsum}
	follows
	and we have
	by \eqref{eq:G:Pieri}, \eqref{eq:H(1):G},
	Lemma \ref{theo:mobius} (1) and
	Lemma \ref{theo:MobiusLem} (with all ordering reversed)
	\begin{align}
	\Gs{\la} G_{(a)} 
	&= \sum_{\nu} c^{\nu}_{\la,(a)} \Gs{\nu}
		= \sum_{\mu\in Q} G_\mu,
	\end{align}
	where $Q$ is the order filter of $\PP$
	generated by 
	$\{\la^{(1)}, \la^{(2)}, \cdots\}$,
	i.e.\,the set of $\mu\in\PP$ satisfying $\mu\supset\la^{(i)}$ for some $i$.
	Hence \eqref{eq:Gs:Pieri:G} follows.
\end{proof}

\begin{proof}[Proof of Lemma \ref{theo:mobius}]
Fix $\la\in\PP$.
Let $r_0<r_1<\dots<r_t$ be the row indices
for which rows there are addable corners of $\la$,
i.e.\,$\la_{r_i-1}>\la_{r_i}$
(we consider $\la_0=\infty$, whence $r_0=1$).
Let $n_i=\la_{r_i-1}-\la_{r_i}$,
i.e.\,the number of boxes that can be added to $\la$ in the $r_i$-th row
(we consider $n_0=\infty$).
Then 
\[
	\hs{\la}
	\simeq
	\{(b_0,\dots,b_t)\in\Z^{t+1}\mid 0\le b_i\le n_i\ (\text{for } 0\le i\le t)\},
\]
where $(b_0,\dots,b_t)$ in the right-hand side corresponds to
the partition obtained by adding $b_i$ boxes to $\la$ in the $r_i$-th row.
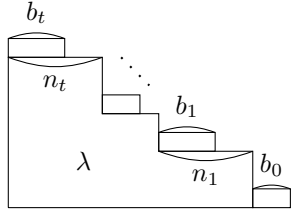
\[
\tikzitem[0.25]{
	\draw (0,0) -| (13,3) -| (8,5) -| (5,8) -| (0,0);
	\draw (13,0) rectangle +(2,1);
	\draw (8,3) rectangle +(3,1);
	\draw (5,5) rectangle +(2,1);
	\draw (0,8) rectangle +(3,1);

	\draw (0,9) to [out=20, in=160] node (bt)[above]{$b_t$} +(3,0);
	\draw (0,8) to [out=-20, in=-160] node [below]{$n_t$} +(5,0);
	
	\draw (8,4) to [out=20, in=160] node (b1)[above]{$b_1$} +(3,0);
	\draw (8,3) to [out=-20, in=-160] node[below]{$n_1$} +(5,0);
	
	\draw [loosely dotted, thick] (6,8) -- (7.5,6.5);

	\draw (13,1) to [out=20, in=160] node[above]{$b_0$} +(2,0);
	
	\node at (4,2.5) {$\la$};
}
\]

Under this correspondence
$\mu\mapsto (b_0,\dots,b_t)$ and
$\nu\mapsto (c_0,\dots,c_t)$,
we have $\nu\subset\mu\iff c_i\le b_i$ (for all $i$) and
\begin{equation}\label{eq:stats}
|\nu/\la| = \sum_{i=0}^{t} c_i, \qquad
r(\nu/\la) = \sum_{i=0}^{t}\DE{c_i>0}, \qquad
r(\la/\bar\nu) = \sum_{i=1}^{t} \DE{c_i<n_i},
\end{equation}
where we use the notation 
$\DE{P}=1$ if $P$ is true and $\DE{P}=0$ if $P$ is false
for a condition $P$.

Now we prove \eqref{eq:c_sum}.
For $\mu\in\hs[\ge a]{\la}$ by \eqref{eq:stats} we have
\begin{align}
\text{(LHS of \eqref{eq:c_sum})} 
&=
\sum_{\substack{\nu\in\hs[\ge a]{\la} \\ \nu\subset\mu}}
	(-1)^{|\nu/\la|-a} \binom{r(\nu/\la)-1}{|\nu/\la|-a} \\
&= 
\sum_{0\le c_0\le b_0}
\sum_{0\le c_1\le b_1}
\dots
\sum_{0\le c_t\le b_t}
\DE{\sum_{i=0}^{t} c_i\ge a}
(-1)^{\sum_{i=0}^{t}c_i - a}
\binom{\sum_{i=0}^{t}\DE{c_i>0}-1}{\sum_{i=0}^{t}c_i-a}.
\label{eq:c_sum:totyu}
\intertext{
	Applying Lemma \ref{theo:binom} below
	to simplify the summation on $c_t$,
	we have
}
&= 
\sum_{0\le c_0\le b_0}
\dots
\sum_{0\le c_{t-1}\le b_{t-1}}
\DE{b_t + \sum_{i=0}^{t-1} c_i\ge a}
(-1)^{b_t+\sum_{i=0}^{t-1}c_i - a}
\binom{\sum_{i=0}^{t-1}\DE{c_i>0}-1}{b_t+\sum_{i=0}^{t-1}c_i - a}.
\intertext{
	Repeating this to simplify the summations on $c_{0},\dots,c_{t-1}$,
	we have
}
&= \dots \\
&= \DE{\sum_{i=0}^{t} b_i\ge a} 
(-1)^{\sum_{i=0}^{t}b_i - a}
\binom{-1}{\sum_{i=0}^{t} b_i - a} \\
&= \DE{\sum_{i=0}^{t} b_i\ge a} 
= \DE{|\mu/\la|\ge a} 
= 1.
\end{align}
Hence \eqref{eq:c_sum} is proved.

Next we prove \eqref{eq:d_sum}.
By similar arguments we have
\begin{align}
	\text{(LHS of \eqref{eq:d_sum})} 
	&= 
		\sum_{b_0\le c_0\le n_0}
		\sum_{b_1\le c_1\le n_1}
		\dots
		\sum_{b_t\le c_t\le n_t}
		\DE{\sum_{i=0}^{t} c_i\le a}
		(-1)^{a-\sum_{i=0}^{t}c_i}
		\binom{\sum_{i=1}^{t}\DE{c_i<n_i}}{a-\sum_{i=0}^{t}c_i}.
		\label{eq:d_sum:totyu}
\end{align}
Note that this is actually a finite sum despite $n_0=\infty$,
and we can replace $n_0$ with a sufficiently large positive integer
without changing the value of \eqref{eq:d_sum:totyu}.
Noticing $\DE{c_0<n_0}=1$ for any $c_0$ that contributes to the summation \eqref{eq:d_sum:totyu},
and letting $b'_i = n_i-b_i$, $c'_i=n_i-c_i$ and $a'=(\sum_{i=0}^{t}n_i)-a$,
we have
\begin{align}
	\eqref{eq:d_sum:totyu}
	&= 
	\sum_{0\le c'_0\le b'_0}
	\sum_{0\le c'_1\le b'_1}
	\dots
	\sum_{0\le c'_t\le b'_t}
	\DE{\sum_{i=0}^{t} c'_i\ge a'}
	(-1)^{\sum_{i=0}^{t}c'_i - a'}
	\binom{\sum_{i=0}^{t}\DE{c'_i>0}-1}{\sum_{i=0}^{t}c'_i-a'}.
	\intertext{
		Since this summation is of the same form as \eqref{eq:c_sum:totyu}, 
		by the same arguments we have
	}
	&=
	\DE{\sum_{i=0}^{t}b'_i\ge a'}
	= \DE{\sum_{i=0}^{t}b_i\le a} = \DE{|\mu/\la|\le a} = 1.
\end{align}
Hence \eqref{eq:d_sum} is proved.
\end{proof}

\begin{lemm}\label{theo:binom}
	For $R,q,b, b'\in\mathbb{Z}$ with $b'\le b$, we have
	\[
		\sum_{b'\le x\le b}
			\DE{x\ge R} 
			(-1)^{x-R}
			\binom{q+\DE{x>b'}}{x-R}
		= \DE{b\ge R} (-1)^{b-R} \binom{q}{b-R},
	\]
	where we use the notation
	$\DE{P}=1$ if $P$ is true and $\DE{P}=0$ if $P$ is false
	for a condition $P$.
\end{lemm}
\begin{proof}
	We carry induction on $b-b'$.
	The lemma is clear when $b'=b$.
	When $b'<b$,
	it is easy to check
	\[
		- \DE{b'\ge R} \binom{q}{b'-R}
		+ \DE{b'+1\ge R} \binom{q+1}{b'+1-R}
		= \DE{b'+1\ge R} \binom{q}{b'+1-R}.
	\]
	Hence we can replace $b'$ with $b'+1$, completing the proof.
\end{proof}


\appendix

\section{M\"obius function of a poset}\label{sect:Prel::Mobius}

For basic definitions for posets
we refer the reader to \cite[Chapter 3]{MR2868112}.

For a locally finite (i.e.\,every interval is finite) poset $P$,
the \textit{M\"obius function} $\mu_{P}(x,y)$ 
(for $x,y\in P$ with $x\le y$) is characterized by
\[
\sum_{x\le z\le y} \mu_{P}(x,z) = \delta_{xy} \quad\text{for any $x\le y$},
\]
or equivalently
\begin{equation}\label{eq:mobius}
\sum_{x\le z\le y} \mu_{P}(z,y) = \delta_{xy} \quad\text{for any $x\le y$}.
\end{equation}

\begin{lemm}\label{theo:MobiusLem}
	Let $\wh{P}$ be a locally finite poset with the maximum element $\hatone$.
	Let $P=\wh{P}\sm\{\hatone\}$
	and $\{x_1,\cdots,x_n\}$ be the maximal elements in $P$,
i.e.\,the coatoms in $\wh{P}$.
	Consider formal variables $\{g(s)\mid s\in \wh{P}\}$ and 
	let 
	$\widetilde{g}(t)=\sum_{s\le t} g(s)$
	for $t\in \wh{P}$.
	
	\noindent $(1)$
	We have
	\begin{align}
		\sum_{s\in P} g(s) 
		= - \sum_{s\in P} \mu_{\wh{P}}(s,\hatone) \wt{g}(s). \label{eq:mobius_gsum2}
	\end{align}
	
	\noindent $(2)$
	Assume that $P$ admits the meet operation $\wedge$.
	Then
	\begin{align}
		\sum_{s\in P} g(s) &= 
		\sum_{m\ge 1} (-1)^{m-1}
		\sum_{i_1<\dots<i_m}
		\wt{g}(x_{i_1}\wedge\cdots\wedge x_{i_m}) \\
		\bigg(&=
		\sum_{i} \wt{g}(x_i)
		- \sum_{i<j} \wt{g}(x_i\wedge x_j)
		+ \sum_{i<j<k} \wt{g}(x_i\wedge x_j\wedge x_k)
		- \dots, \bigg)
	\end{align}

	\noindent $(3)$
	In the same situation as $(2)$,	
	$\mu_{\wh{P}}(s,\hatone) = 0$
	unless $s$ is of the form $s=x_{i_1}\wedge\dots\wedge x_{i_l}$,
	and
	\begin{equation}\label{eq:mobius_PP}
	\mu_{\wh{P}}(s,\hatone) = \mu_{\wh{P}'}(s,\hatone)
	\end{equation}
	for any subposet $\wh{P}'$ of $\wh{P}$ that contains
	all elements of the form
	$x_{i_1}\wedge\dots\wedge x_{i_l}$
	(including $\hatone$ as the meet of an empty set).
	
\end{lemm}
\begin{proof}
	It is known (see \cite[Proposition 3.7.1]{MR2868112} for example) that
	\begin{equation}\label{eq:mobius_inv}
	g(t)=\sum_{s\le t} \mu_{\wh{P}}(s,t) \widetilde{g}(s)
	\quad\text{(for $\forall t\in \wh{P}$)}.
	\end{equation}
	Hence we have
	\begin{align}
	\sum_{s\in P} g(s) 
	= \widetilde{g}(\hatone) - g(\hatone) 
	= \wt{g}(\hatone) - \sum_{s\in\wh{P}} \mu_{\wh{P}}(s,\hatone) \wt{g}(s) 
	= - \sum_{s\in P} \mu_{\wh{P}}(s,\hatone) \wt{g}(s), \label{eq:mobius_gsum}
	\end{align}
	proving (1).	
	(2) is by the Inclusion-Exclusion Principle.
	(3) follows from (1) and (2).
\end{proof}

\bibliographystyle{abbrv}

\end{document}